\documentclass{article}
\usepackage{amsmath}
\usepackage{amssymb}

\newtheorem{theorem}{Theorem}

\newtheorem{lemma}[theorem]{Lemma}

\newtheorem{proposition}[theorem]{Proposition}

\newenvironment{proof}[1][Proof]{\noindent\textbf{#1.} }{\ \rule{0.5em}{0.5em}}
\input{tcilatex}

\begin{document}

\title{Local behaviour of the remainder in Renewal theory}
\author{Doney, Ron}

\begin{abstract}
Several terms in an asymptotic estimate for the renewal mass function in a
discrete random walk which has positive mean and regularly varying
right-hand tail are given. Similar results are given for the renewal density
in the absolutely continuous case.
\end{abstract}

\maketitle

\section{Introduction}

The context is that $S=(S_{n},n\geq 0)$ is a random walk on the line, with $%
S_{0}=0$ and $S_{n}=\tsum\nolimits_{1}^{n}X_{r}$ for $n\geq 1,$ where $X_{r}$
are i.i.d. with mass function $p$ or density function $f.$ The object of our
study is the renewal mass function, (respectively density function), which
we denote by $u,$ viz%
\begin{eqnarray*}
u_{m} &=&\sum_{0}^{\infty }p_{n}(m),\text{ with }p_{n}(m)=\text{ }P(S_{n}=m),%
\text{ }m=0,\pm 1,\pm 2, \\
\text{or }u(x) &=&\sum_{0}^{\infty }f_{n}(x),\text{ with }f_{n}\text{ the
density function of }S_{n}.
\end{eqnarray*}%
We assume that $\mu =ES_{1}\in (0,\infty ),$ $ES_{1}^{2}=\infty $ and that
for some $\alpha \in (1,2)$

\begin{equation}
\overline{F}(n):=P(S_{1}>n)\in RV(-\alpha ).  \label{r1}
\end{equation}%
For the most part we will assume also that the lefthand tail is neglible in
comparision with the righthand tail, viz%
\begin{equation}
F(-n):=P(S_{1}\leq -n)=o(\overline{F}(n))\text{ as }n\rightarrow \infty .
\label{r0}
\end{equation}%
Our objective is to give the most detailed possible description of the
asymptotic behaviour of the remainder%
\begin{equation}
d_{n}:=u_{n}-\frac{1}{\mu }\text{ or }d(x):=u(x)-\frac{1}{\mu },  \label{r2}
\end{equation}%
as $n,x\rightarrow \infty ,$ without making any additional restrictions. The
most relevant known results are in \cite{ro}; in the context of an
integer-valued renewal process but with (\ref{r1}) replaced by the
assumption that $\overline{F}(n)$ is either $O(r_{n})$ or $o(r_{n}),$ where $%
r_{n}\in RV(-\alpha ),$ first order $O$ and $o$ estimates for $d_{n}$ were
given. The method used therein involves Banach Algebra, and cannot be used
to deal with exact asymptotics. The method we use is technically simpler,
and could be adapted to give essentially all of the results in \cite{ro}.

\section{The lattice case}

Assume that we are not in the renewal situation, so that $\mu
_{-}:=\tsum\nolimits_{-\infty }^{-1}mp_{m}<0,$ and with $\mu
_{+}:=\tsum\nolimits_{1}^{\infty }mp_{m}$ introduce two probability mass
functions%
\begin{equation*}
\phi _{n}^{+}=\left\{ 
\begin{array}{ccc}
\mu _{+}^{-1}\overline{F}(n) & \text{for} & n>0, \\ 
0 & \text{for} & n\leq 0,%
\end{array}%
\right.
\end{equation*}%
and%
\begin{equation*}
\phi _{n}^{-}=\left\{ 
\begin{array}{ccc}
-\mu _{-}^{-1}F(n) & \text{for} & n<0, \\ 
0 & \text{for} & n\geq 0,%
\end{array}%
\right.
\end{equation*}%
then we put%
\begin{equation*}
\phi _{n}=\frac{\mu _{+}\phi _{n}^{+}}{\mu }+\frac{\mu _{-}\phi _{n}^{-}}{%
\mu },n=0,\pm 1,\pm 2,\cdots .
\end{equation*}%
Note that $\phi $ is negative on the negative axis, but $\tsum\nolimits_{-%
\infty }^{\infty }\phi _{n}=1.$ Moreover, a careful summation by parts shows
that%
\begin{equation}
\hat{\phi}(t):=\tsum\nolimits_{-\infty }^{\infty }\phi _{n}e^{int}=\frac{1-%
\hat{p}(t)}{\mu (1-e^{it})}.  \label{r11}
\end{equation}%
So we can formally write%
\begin{eqnarray*}
\hat{u}(t) &:&=\tsum\nolimits_{-\infty }^{\infty }u_{n}e^{int}=\frac{1}{1-%
\hat{p}(t)}=\frac{1}{\mu (1-e^{it})\hat{\phi}(t)} \\
&=&\frac{1}{\mu (1-e^{it})}+\frac{1-\hat{\phi}(t)}{\mu (1-e^{it})}+\frac{(1-%
\hat{\phi}(t))^{2}}{\mu (1-e^{it})}+\cdots
\end{eqnarray*}%
and undoing the transform suggests the expansion%
\begin{equation}
u_{n}=\frac{1}{\mu }\left( \boldsymbol{1}_{\{n\geq 0\}}+\overline{\Phi }%
_{1,n}+\overline{\Phi }_{2,n}+\cdots \right) ,  \label{r22}
\end{equation}%
where%
\begin{equation*}
\overline{\Phi }_{1,n}=\left\{ 
\begin{array}{ccc}
\tsum\nolimits_{n+1}^{\infty }\phi _{r} & \text{for} & n\geq 0, \\ 
-\tsum\nolimits_{-\infty }^{n}\phi _{r} & \text{for} & n\,<0.%
\end{array}%
\right. \text{ and }\overline{\Phi }_{k+1,n}=\overline{\Phi }_{k,n}-(%
\overline{\Phi }\ast \phi )_{n},\text{ }k\geq 1.
\end{equation*}%
(The fact that $\frac{1-\varphi (t)}{1-e^{it}}=\tsum\nolimits_{-\infty
}^{\infty }\overline{\Phi }_{1,n}e^{int}$ can be checked by summation by
parts.) We will investigate how far the expansion (\ref{r22}) is valid and
give the asymptotic behaviour of its terms under the standing assumptions
that $ES_{1}=\mu \in (0,\infty ),$ $ES_{1}^{2}=\infty ,$ and (\ref{r1}) and (%
\ref{r0}) hold. Clearly then $\overline{\Phi }_{1,n}\backsim n\overline{F}%
(n)/(\beta \mu _{+})\in RV(-\beta )$ where $\beta =\alpha -1.$

\begin{theorem}
\label{D}(i) Under these assumptions if $k\geq 2$ and $k\beta <\alpha $ it
holds that as $n\rightarrow \infty $ 
\begin{equation}
\overline{\Phi }_{k,n}\backsim c(k,\beta )\left( \overline{\Phi }%
_{1,n}\right) ^{k},  \label{r51}
\end{equation}%
where%
\begin{equation}
c(k,\beta )=\frac{(1-k\beta )\Gamma (1-\beta )^{k}}{\Gamma (2-k\beta )}%
=\left\{ 
\begin{array}{ccc}
\frac{\Gamma (1-\beta )^{k}}{\Gamma (1-k\beta )} & \text{if} & k\beta \neq 1,
\\ 
0 & \text{if} & k\beta =1.%
\end{array}%
\right.  \label{r53}
\end{equation}

(ii) If $r^{\ast }=\max (k:k\beta <\alpha )$ write 
\begin{equation*}
u_{n}=\frac{1}{\mu }\left( \boldsymbol{1}_{\{n\geq 0\}}+\overline{\Phi }%
_{1,n}+\cdots +\overline{\Phi }_{r^{\ast },n}+E^{\ast }\right) .
\end{equation*}%
Then 
\begin{equation}
E^{\ast }=o(\left( \overline{\Phi }_{1,n}\right) ^{r^{\ast }})\text{ as }%
n\rightarrow \infty .  \label{r26}
\end{equation}

(iii) If $\alpha =2$ define for $n\geq 1$ a slowly varying function by $\mu
_{2}(n)=\sum_{1}^{n}r\phi _{r}$: then%
\begin{equation*}
u_{n}-\frac{1}{\mu }-\frac{\overline{\Phi }_{1,n}}{\mu }\backsim \frac{-\phi
_{n}\mu _{2}(n)}{\mu }\text{ as }n\rightarrow \infty .
\end{equation*}
\end{theorem}

Since $\psi _{k}(t):=E(e^{it\overline{\Phi }_{k,n}})=(1-)^{k}/(1-e^{it})$
our first step is to note that the asymptotic behaviour of both $1-\hat{\phi}%
(t)$ and $\hat{\phi}^{\prime }(t)$ follow from known results.

\begin{lemma}
Assume that for some $\alpha \in (1,2)$ we have $\overline{F}\in RV(-\alpha
) $ and $F(-x)/\overline{F}(x)\rightarrow 0$ as $x\rightarrow \infty ,$ so
that as $n\rightarrow \infty $ 
\begin{equation*}
\overline{\Phi }_{1,n}\backsim \frac{n\overline{F}(n)}{\beta \mu }\text{ and 
}\overline{\Phi }_{1,-n}=o(\overline{\Phi }_{1,n}).
\end{equation*}%
Define $r(t)=\overline{\Phi }_{1,1/t},$ then as $t\downarrow 0,$

(i) 
\begin{equation}
1-\hat{\phi}(t)\backsim \Gamma (1-\beta )e^{-i\pi \beta /2}r(t),\text{ }
\label{r32}
\end{equation}%
and

(ii)%
\begin{equation}
\hat{\phi}^{\prime }(t)\backsim \beta \Gamma (1-\beta )e^{-i\pi \beta
/2}t^{-1}r(t).  \label{r33}
\end{equation}
\end{lemma}

\begin{proof}
Results in \cite{P} tell us that as $t\downarrow 0$ 
\begin{equation*}
\,1-\hat{p}(t)+i\mu t\backsim \frac{\pi e^{-i\pi \alpha /2}\overline{F}(1/t)%
}{\Gamma (\alpha )\sin \alpha \pi }\text{ and }\hat{p}^{\prime }(t)-i\mu
\backsim \alpha \frac{\pi e^{-i\pi \alpha /2}\overline{F}(1/t)}{t\Gamma
(\alpha )\sin \alpha \pi }.
\end{equation*}%
Then (\ref{r11}) gives%
\begin{eqnarray*}
1-\hat{\phi}(t) &=&\frac{\hat{p}(t)-1+\mu (1-e^{it})}{\mu (1-e^{it})} \\
&=&\frac{\hat{p}(t)-1-it\mu +O(t^{2})}{\mu (1-e^{it})}\backsim \frac{-i\pi
e^{-i\pi \alpha /2}\overline{F}(1/t)}{\mu \Gamma (\alpha )t\sin \alpha \pi }
\\
&=&\frac{i\pi e^{-i\pi \alpha /2}\overline{F}(1/t)}{\mu \beta \Gamma (\beta
)t\sin \beta \pi }=\frac{\pi e^{-i\pi \beta /2}\overline{F}(1/t)}{\mu \beta
\Gamma (\beta )t\sin \beta \pi }=\frac{\Gamma (1-\beta )e^{-i\pi \beta
/2}r(t)}{\mu }
\end{eqnarray*}%
Note that we here we have used the identity%
\begin{equation}
\frac{\pi }{\Gamma (\beta )\sin \pi \beta }=\Gamma (1-\beta ).  \label{r36}
\end{equation}%
Differentiating (\ref{r11}) and using the second estimate gives (ii).
\end{proof}

We split the proof into 2 cases

\begin{proposition}[A]
Suppose that $k\geq 2$ and $k\beta \leq 1:$ then it holds that as $%
n\rightarrow \infty $ 
\begin{equation}
\overline{\Phi }_{k,n}\backsim c(k,\beta )\left( \overline{\Phi }%
_{1,n}\right) ^{k}.  \label{r31}
\end{equation}
\end{proposition}

\begin{proof}
\ Note first that the asymptotic behaviour of $\psi _{k}$ and $\psi
_{k}^{\prime }$follow from (\ref{r32}) and (\ref{r33}), and we see that if
we set $\rho (t)=r(t)^{k}/t$ then $\rho \in RV(-\gamma )$ at zero$,$ with $%
\gamma =1-k\beta $ and as $t\downarrow 0$ 
\begin{equation}
\psi _{k}(t)\backsim i(c_{1}+ic_{2})\rho (t),\text{ and }\psi _{k}^{\prime
}(t)\backsim -i(c_{1}+ic_{2})\gamma \rho (t)/t,  \label{r35}
\end{equation}%
where, from (\ref{r32})%
\begin{equation*}
c_{1}+ic_{2}=\Gamma (1-\beta )^{k}e^{-i\pi k\beta /2}.
\end{equation*}%
Note that $\func{Re}\psi _{k}$ is even and $\func{Im}\psi _{k}$ is odd, so
that%
\begin{equation*}
\overline{\Phi }_{k,n}=\frac{1}{\pi }\int_{0}^{\pi }\left( \func{Re}\psi
_{k}(u)\cos nu+\func{Im}\psi _{k}(u)\sin nu\right) du
\end{equation*}%
Writing $n^{-1}\pi =\varepsilon $ we see that for any fixed $K,$ as $%
n\rightarrow \infty $%
\begin{eqnarray*}
&&\func{Re}\frac{1}{\pi }\int_{0}^{K\varepsilon }\psi _{k}(u)e^{-inu}du \\
&=&\frac{1}{n\pi }\int_{0}^{K\pi }\left( \func{Re}\psi _{k}(v/n)\cos v\pi +%
\func{Im}\psi _{k}(v/n)\sin v\pi \right) dv \\
&\backsim &\frac{\Gamma (1-\beta )^{k}}{n\pi }\rho (\frac{1}{n}%
)\int_{0}^{K\pi }v^{-\gamma }\left( \cos v\pi \sin \pi k\beta /2+\sin v\pi
\cos \pi k\beta /2\right) dv.
\end{eqnarray*}%
Also integration by parts gives%
\begin{eqnarray*}
&&\int_{|u|>K\varepsilon }\psi _{k}(u)e^{-inu}du \\
&=&\frac{1}{-in}\left\{ \psi _{k}(K\varepsilon )e^{-iK\pi }-\psi
_{k}(-K\varepsilon )e^{iK\pi }-\int_{|u|>K\varepsilon }\psi _{k}^{\prime
}(u)e^{-inu}du\right\}  \\
&=&\frac{1}{-in}\left\{ \psi _{k}(K\pi /n)e^{-iK\pi }-\psi _{k}(-K\pi
/n)e^{iK\pi }+\frac{1}{in}\int_{|u|>K\pi }\psi _{k}^{\prime
}(v/n)e^{-iv}dv\right\}  \\
&\backsim &\frac{\psi _{k}(1/n)}{-in}O\left( \int_{|u|>K\pi }v^{\gamma
-2}e^{-iv}dv+(K\pi )^{\gamma -1}\right) ,
\end{eqnarray*}%
and letting $n\rightarrow \infty $ then $K\rightarrow \infty $ we see that%
\begin{eqnarray*}
\overline{\Phi }_{k,n} &\backsim &\frac{\Gamma (1-\beta )^{k}}{\pi n}\rho (%
\frac{1}{n})\int_{0}^{\infty }v^{-\gamma }\left( \cos v\pi \sin \pi k\beta
/2+\sin v\pi \cos \pi k\beta /2\right) dv \\
&=&\frac{2\Gamma (1-\beta )^{k}}{\pi n}\rho (\frac{1}{n})\Gamma (k\beta
)\Gamma (1-\beta )^{k}\cos \frac{k\beta \pi }{2}\sin \frac{k\beta \pi }{2} \\
&=&\frac{1}{\pi n}\rho (\frac{1}{n})\Gamma (k\beta )\Gamma (1-\beta
)^{k}\sin k\beta \pi  \\
&=&\left\{ 
\begin{array}{ccc}
(\overline{\Phi }_{1,n})^{k}\frac{\Gamma (1-\beta )^{k}}{\Gamma (1-k\beta )}
& \text{if} & k\beta \neq 1, \\ 
0 & \text{if} & k\beta =1.%
\end{array}%
\right. .
\end{eqnarray*}%
(We have used the identity (\ref{r36}) with $\beta $ replaced by $k\beta .)$
\end{proof}

\begin{proposition}
If now $k\beta \in (1,\alpha )$ the result (\ref{r31}) also holds.
\end{proposition}

\begin{proof}
We start with an integration by parts to get

\begin{equation*}
\func{Re}\int_{-\pi }^{\pi }\psi _{k}(u)e^{-inu}du=\func{Re}\frac{i}{n}%
\int_{-\pi }^{\pi }\psi _{k}^{\prime }(u)e^{-inu}du,
\end{equation*}%
the contribution from the end-points vanishing due to periodicity. Since $%
\psi _{k}^{\prime }(u)\backsim (k\beta -1)\psi _{k}(u)/u\in RV(\gamma -1)$
where now $\gamma =k\beta -1\in (0,1),$ we can repeat the previous proof to
get, with $\rho (t)=r(t)^{k}/t$

\begin{equation*}
\lim_{K\rightarrow \infty }\lim_{n\rightarrow \infty }\frac{i}{2\pi \rho
(1/n)}\int_{|u|\leq \varepsilon K}\psi _{k}^{\prime }(v)e^{-inv}dv=(1-k\beta
)\frac{\Gamma (1-\beta )^{k}}{\Gamma (2-k\beta )},
\end{equation*}%
and we will now show that 
\begin{equation*}
\lim_{K\rightarrow \infty }\lim_{n\rightarrow \infty }\frac{i}{\rho (1/n)}%
\int_{K\varepsilon <|u|\leq \pi }\psi _{k}^{\prime }(u)e^{-inu}du=0.
\end{equation*}%
We have%
\begin{eqnarray*}
\mu \psi _{k}^{\prime }(u) &=&\frac{k(1-\hat{\phi}(u)^{k-1}\hat{\phi}%
^{\prime }(u)+ie^{iu}(1-\hat{\phi}(u)^{k}/(1-e^{iu})}{1-e^{iu}}=\varphi
_{1}(u)+\varphi _{2}(u),\text{ where} \\
\varphi _{1}(u) &=&\frac{k(1-\hat{\phi}(u)^{k-1}\hat{\phi}^{\prime }(u)}{%
1-e^{iu}},\text{ }\varphi _{2}(u)=\frac{ie^{iu}(1-\hat{\phi}(u)^{k}}{%
(1-e^{iu})^{2}}.
\end{eqnarray*}%
Notice that $\varphi _{2}$ is differentiable and periodic, and as $%
u\downarrow 0,$ $\varphi _{2}(u)\backsim -\psi _{k}(u)/u\in RV(-\gamma ),$
so we can repeat the previous proof to get%
\begin{equation*}
\lim_{K\rightarrow \infty }\lim_{n\rightarrow \infty }\frac{i}{\rho (1/n)}%
\int_{K\varepsilon <|u|\leq \pi }\varphi _{2}(u)e^{-inu}du=0.
\end{equation*}%
Recalling that $\varepsilon =\pi /n,$ we write 
\begin{eqnarray}
\int_{K\varepsilon <|u|\leq \pi }\varphi _{1}(u+\varepsilon )e^{-inu}du
&=&e^{i\pi }\int_{(K-1)\varepsilon }^{\pi -\varepsilon }\varphi
_{1}(v)e^{-inv}dv+e^{i\pi }\int_{-\pi -\varepsilon \varphi
_{1}}^{-\varepsilon (K+1)}e^{-inv}dv  \notag \\
&=&-\int_{K\varepsilon <|u|\leq \pi }\varphi
_{1}(v)e^{-inv}dv-\int_{-K\varepsilon }^{-\varepsilon (K+1)}\varphi
_{1}(v)e^{-inv}dv,  \notag \\
&:&=-\int_{K\varepsilon <|u|\leq \pi }\varphi
_{1}(v)e^{-inv}dv-E_{\varepsilon }  \notag \\
\text{so that}\int_{K\varepsilon <|u|\leq \pi }\varphi _{1}(v)e^{-inv}dv
&=&\int_{K\varepsilon <|u|\leq \pi }\left( \varphi _{1}(u)-\varphi
_{1}(u+\varepsilon )\right) e^{-inu}du+2E_{\varepsilon }  \label{r40}
\end{eqnarray}%
We know that $|\varphi _{1}(u)|\backsim c\rho (u)$ as $u\downarrow 0,$ so 
\begin{equation*}
|E_{\varepsilon }|\leq \int_{(K-1)\varepsilon }^{(K+1)\varepsilon }|\varphi
_{1}(v)|dv\backsim 2\varepsilon |\varphi _{1}(K\varepsilon )|\backsim
cK^{k\beta -2}\rho (\varepsilon ),
\end{equation*}%
and we see that $\lim_{K\rightarrow \infty }\lim \sup_{\varepsilon
\downarrow 0}\frac{|E_{\varepsilon }|}{\rho (\varepsilon )}=0.$ We will
finish the proof by showing that%
\begin{equation}
\lim_{K\rightarrow \infty }\lim \sup_{\varepsilon \downarrow 0}\frac{1}{\rho
(\varepsilon )}\int_{K\varepsilon <|u|\leq \pi }\left( \varphi
_{1}(u)-\varphi _{1}(u+\varepsilon )\right) e^{-inu}du=0.  \label{r41}
\end{equation}%
Note first that a small calculation shows that it suffices to deal with the
contribution from $(K\varepsilon ,\pi ]$ in (\ref{r41}). First we show that
for any fixed $\delta >0$ the contribution from $\delta <u\leq \pi $ to the
integral in (\ref{r41}) is neglible. Note that $\phi _{1}$ and $\phi
_{1}^{\prime }$ are each bounded in modulus by a constant on this interval.
Moreover we easily see from the following Lemma \ref{A} that on the range of
integration 
\begin{equation*}
|\hat{\phi}^{\prime }(u)-\hat{\phi}^{\prime }(u+\varepsilon )|\leq c\left( |%
\hat{p}^{\prime }(u)-\hat{p}^{\prime }(u+\varepsilon )|+|\hat{\phi}(u)-\hat{%
\phi}(u+\varepsilon )|\right) \leq c\varepsilon ,
\end{equation*}%
so that%
\begin{equation*}
|\int_{\delta <|u|\leq \pi }\left( \varphi _{1}(u)-\varphi
_{1}(u+\varepsilon )\right) e^{-inu}du|=O(\varepsilon ),
\end{equation*}%
so we are left to show that for sufficiently small $\delta $%
\begin{equation*}
\lim_{K\rightarrow \infty }\lim \sup_{\varepsilon \downarrow 0}\frac{%
|J_{\varepsilon }|}{\rho (\varepsilon )}=0,\text{ where}
\end{equation*}%
\begin{equation*}
J_{\varepsilon }:=\int_{u\in (K\varepsilon ,\delta )}|\varphi
_{1}(u)-\varphi _{1}(u+\varepsilon )|du.
\end{equation*}
\end{proof}

\begin{lemma}
(i) 
\begin{eqnarray*}
\sup_{t}|\hat{p}(t)-\hat{p}(t+\varepsilon )-i\mu \varepsilon |
&=&O(\varepsilon r(\varepsilon ))\text{ and } \\
\sup_{t}|\hat{\phi}(t)-\hat{\phi}(t+\varepsilon )| &=&O(r(\varepsilon )).%
\text{as }\varepsilon \downarrow 0.\text{ }
\end{eqnarray*}

(ii) For any fixed $\delta >0$%
\begin{equation*}
\sup_{\delta <t\leq \pi }|\hat{\phi}(t)-\hat{\phi}(t+\varepsilon
)|=O(\varepsilon )\text{ as }\varepsilon \downarrow 0.\text{ }
\end{equation*}

(iii) As $\varepsilon ,t\downarrow 0$ with $t\geq 2\varepsilon ,$%
\begin{equation*}
|\hat{\phi}(t)-\hat{\phi}(t+\varepsilon )|=O(\frac{\varepsilon r(t)}{t})
\end{equation*}
\end{lemma}

\begin{proof}
(i) Just write%
\begin{eqnarray*}
\sup_{t}|\hat{p}(t)-\hat{p}(t+\varepsilon )-i\mu \varepsilon | &\leq
&\sum_{-\infty }^{\infty }p_{n}|1-e^{in\varepsilon }-in\varepsilon | \\
\leq c\left( \varepsilon ^{2}\sum_{|n|\leq 1/\varepsilon
}n^{2}p_{n}+\varepsilon \sum_{|n|>1/\varepsilon }np_{n}\right) &\backsim &c%
\overline{F}(1/\varepsilon )\backsim c\varepsilon r(\varepsilon ).
\end{eqnarray*}

Similarly%
\begin{eqnarray*}
\sup_{t}|\hat{\phi}(t)-\hat{\phi}(t+\varepsilon )| &\leq &\sum_{-\infty
}^{\infty }\phi _{n}|1-e^{in\varepsilon }| \\
\leq c\left( \varepsilon \sum_{|n|\leq 1/\varepsilon }n\phi
_{n}+\sum_{|n|>1/\varepsilon }\phi _{n}\right) &\backsim &c\overline{\Phi }%
_{1,1/\varepsilon }\backsim cr(\varepsilon ).
\end{eqnarray*}%
(ii) We write%
\begin{eqnarray*}
\hat{\phi}(t)-\hat{\phi}(t+\varepsilon )=1-\hat{\phi}(t+\varepsilon )-(1-%
\hat{\phi}(t)) && \\
=\frac{\hat{p}(t+\varepsilon )-1-\mu (1-e^{i(t+\varepsilon )})}{\mu
(1-e^{i(t+\varepsilon )})}-\frac{\hat{p}(t)-1-\mu (1-e^{it})}{\mu (1-e^{it})}
&& \\
=\left( \hat{p}(t+\varepsilon )-1-\mu (1-e^{i(t+\varepsilon )})\right) \frac{%
e^{it}(1-e^{i\varepsilon })}{\mu (1-e^{it})(1-e^{i(t+\varepsilon )})} && \\
+\frac{\hat{p}(t+\varepsilon )-\hat{p}(t)+\mu (e^{it}-e^{i(t+\varepsilon )})%
}{\mu (1-e^{it})}:=T_{1}+T_{2}. &&
\end{eqnarray*}

Since $|1-e^{it}|$ is bounded away from $0$ on the range, it is clear that $%
T_{1}=O(\varepsilon ),$ and writing%
\begin{equation*}
T_{2}=\frac{\hat{p}(t)-\hat{p}(t+\varepsilon )-i\mu \varepsilon }{1-e^{it}}+%
\frac{(i\varepsilon +e^{it}(1-e^{i\varepsilon })}{(1-e^{it})}
\end{equation*}%
it follows from part (i) that $T_{2}=O(\varepsilon ).$

(iii) Just note that%
\begin{equation*}
T_{1}\backsim (\varepsilon +t)r(\varepsilon +t)\cdot \frac{i\varepsilon }{%
t(\varepsilon +t)}\leq c\frac{\varepsilon r(t)}{t},
\end{equation*}%
and 
\begin{eqnarray*}
|T_{2}| &\leq &\frac{\varepsilon r(\varepsilon )}{|1-e^{it}|}+\frac{%
|e^{it}(1-e^{i\varepsilon }+i\varepsilon )+i\varepsilon (1-e^{it})|}{%
|1-e^{it}|} \\
&\backsim &\frac{\varepsilon r(\varepsilon )}{t}+\frac{\varepsilon
^{2}+\varepsilon t}{t}=O(\frac{\varepsilon r(t))}{t}
\end{eqnarray*}

Rewriting $\varphi _{1}(u)=\frac{g_{1}(u)\phi ^{\prime }(u)}{1-e^{iu}}$ we
show first that replacing $e^{iu}$ by $e^{i(u+\varepsilon )}$ has a neglible
effect. We have%
\begin{eqnarray*}
|\frac{1}{1-e^{iu}}-\frac{1}{1-e^{i(u+\varepsilon )}}| &=&|\frac{%
e^{iu}(1-e^{i\varepsilon })}{(1-e^{i\mu u})(1-e^{i\mu (u+\varepsilon )})}| \\
\leq \frac{C\varepsilon }{|(1-e^{iu})(1-e^{i(u+\varepsilon )})|} &\backsim &%
\frac{C\varepsilon }{u(u+\varepsilon )}\text{ as }u,\varepsilon \downarrow 0.
\end{eqnarray*}%
We know $g_{1}(u)\phi ^{\prime }(u)\backsim \mathbf{c}\rho (u)$ as $%
u\downarrow 0,$ so for sufficiently small $\delta $ 
\begin{eqnarray*}
&\int_{K\varepsilon <|u|\leq \delta }&|\frac{g_{1}(u)\phi ^{\prime }(u)}{%
1-e^{i\mu u}}-\frac{g_{1}(u)\phi ^{\prime }(u)}{1-e^{i\mu (u+\varepsilon )}}%
|du \\
&\leq &c\int_{K\varepsilon <|u|\leq \delta }\frac{|g_{1}(u)\phi ^{\prime
}(u)|du}{|u(u+\varepsilon )|}=c\int_{K<|v|\leq \delta /\varepsilon }\frac{%
|g_{1}(\varepsilon v)\phi ^{\prime }(\varepsilon u)|dv}{|v(v+1)|} \\
&\backsim &c\rho (\varepsilon )\int_{K<|v|\leq \delta /\varepsilon }\frac{%
|v|^{k\beta -1}dv}{|v(v+1)|}\text{ as }\varepsilon \downarrow 0,
\end{eqnarray*}%
where we use Potter's bounds to bound the integrand. Letting $K\rightarrow
\infty $ gives the required result. If we put $d_{\varepsilon }(u)$ for $%
\hat{\phi}(u)-\hat{\phi}(u+\varepsilon ),$ we have%
\begin{equation}
(1-\hat{\phi}(u))^{k-1}-(1-\hat{\phi}(u+\varepsilon ))^{k-1}=d_{\varepsilon
}(u)\sum_{0}^{k-2}(1-\hat{\phi}(u))^{j}(1-\hat{\phi}(u+\varepsilon ))^{k-2-j}
\label{r38}
\end{equation}%
and Lemma \ref{A} gives $d_{\varepsilon }(u)=O(\varepsilon r(u)/u)$ on the
range of integration. So for each fixed $j$%
\begin{eqnarray*}
&&\int_{K\varepsilon <u\leq \delta }\frac{d_{\varepsilon }(u)(1-\hat{\phi}%
(u))^{j}}{1-e^{i(u+\varepsilon )}}\hat{\phi}^{\prime }(u)du \\
&=&O\left( \varepsilon \int_{K\varepsilon <u\leq \delta }\frac{r(u)(1-\hat{%
\phi}(u))^{j}(1-\hat{\phi}(u+\varepsilon ))^{k-2-j}\hat{\phi}^{\prime }(u)}{%
1-e^{i(u+\varepsilon )}}du\right) \\
&=&O\left( \varepsilon ^{2}\int_{K\varepsilon <u\leq \delta }\frac{%
r(\varepsilon u)(1-\hat{\phi}(\varepsilon u))^{j}(1-\hat{\phi}(\varepsilon
(u+1))^{k-2-j}\hat{\phi}^{\prime }(\varepsilon u)}{1-e^{i\varepsilon (u+1)}}%
du\right) \\
&=&O\left( \frac{r(\varepsilon )^{k}}{\varepsilon }\int_{K}^{\infty }\frac{%
u^{\beta (j+1)}(u+1)^{\beta (k-2-j}}{u^{2}(u+1)}du\right)
\end{eqnarray*}%
and it is easy to see that%
\begin{equation*}
\lim_{K\rightarrow \infty }\lim \sup_{\varepsilon \downarrow 0}\frac{1}{\rho
(\varepsilon )}\int_{K\varepsilon <|u|\leq \delta }\frac{\left(
g_{1}(u)-g_{2}(u+\varepsilon )\right) }{1-e^{i(u+\varepsilon )}}du=0.
\end{equation*}%
Finally we deal with%
\begin{eqnarray}
&&|\int_{K\varepsilon <|u|\leq \delta }\frac{(1-\hat{\phi}(u+\varepsilon
))^{k-1}}{1-e^{i(u+\varepsilon )}}\left( \hat{\phi}^{\prime }(u)-\hat{\phi}%
^{\prime }(u+\varepsilon )\right) e^{-inu}du|  \notag \\
&\leq &c\int_{K\varepsilon <|u|\leq \delta }\frac{r(u+\varepsilon )^{k-1}}{%
u+\varepsilon }|\hat{\phi}^{\prime }(u)-\hat{\phi}^{\prime }(u+\varepsilon
)|du|  \label{r42}
\end{eqnarray}%
We have 
\begin{equation*}
\mu \hat{\phi}^{\prime }(u)=\frac{-\hat{p}^{\prime }(u)}{1-e^{iu}}+\frac{%
ie^{iu}(1-\hat{p}(u))}{(1-e^{iu})^{2}}
\end{equation*}%
so%
\begin{eqnarray}
\mu (\hat{\phi}^{\prime }(u)-\hat{\phi}^{\prime }(u+\varepsilon )) &=&\left( 
\frac{\hat{p}^{\prime }(u+\varepsilon )}{1-e^{i(u+\varepsilon )}}-\frac{\hat{%
p}^{\prime }(u)}{1-e^{iu}}\right)  \label{r43} \\
&&+\left( \frac{ie^{iu}(1-\hat{p}(u))}{(1-e^{iu})^{2}}-\frac{%
ie^{i(u+\varepsilon )}(1-\hat{p}(u+\varepsilon ))}{(1-e^{i(u+\varepsilon
})^{2}}\right) .
\end{eqnarray}%
We estimate the first term in (\ref{r43}), as $u,\varepsilon \downarrow 0,$
by%
\begin{eqnarray}
&&\frac{\hat{p}^{\prime }(u+\varepsilon )-\hat{p}^{\prime }(u)}{%
1-e^{i(u+\varepsilon )}}+\hat{p}^{\prime }(u)\left( \frac{1}{%
1-e^{i(u+\varepsilon )}}-\frac{1}{1-e^{iu}}\right)  \notag \\
&\backsim &\frac{ir(\varepsilon )}{u+\varepsilon }+i\mu \frac{i\varepsilon }{%
(u+\varepsilon )u}  \label{46}
\end{eqnarray}%
and the second by%
\begin{eqnarray}
&&(1-\hat{p}(u))\left( \frac{ie^{iu}}{(1-e^{iu})^{2}}-\frac{%
ie^{i(u+\varepsilon )}}{(1-e^{i(u+\varepsilon )})^{2}}\right) +\frac{%
ie^{i(u+\varepsilon )}(\hat{p}(u+\varepsilon )-\hat{p}(u))}{%
(1-e^{i(u+\varepsilon })^{2}}  \notag \\
&\backsim &i\mu u\frac{\varepsilon (2u+\varepsilon )}{u^{2}(u+\varepsilon
)^{2}}+O\left( \frac{\varepsilon }{(u+\varepsilon )^{2}}\right) .
\label{r47}
\end{eqnarray}%
Of the 4 terms in (\ref{46}) and (\ref{r47}), all except the first can be
dealt with by previous arguments, but $\frac{r(\varepsilon )}{u+\varepsilon }
$ requires care. Since $k\beta <\alpha $ implies $(k-1)\beta <1$ we see that%
\begin{equation*}
r(\varepsilon )|\int_{K\varepsilon <u\leq \delta }\frac{(1-\hat{\phi}%
(u))^{k-1}du}{(1-e^{i(u+\varepsilon )})(u+\varepsilon )}|\backsim \frac{%
r(\varepsilon )^{k}}{\varepsilon }\int_{K}^{\infty }\frac{u^{(k-1)\beta }}{%
(u+1)^{2}}du,
\end{equation*}%
and the conclusion follows by letting $K\rightarrow \infty .$ This proves
(i), and we start the proof of (ii) by getting a good asymptotic bound for
the difference $u_{n-1}-u_{n}.$ The first step is
\end{proof}

\begin{proposition}
\label{A}With $\Delta _{n}=u_{n-1}-u_{n},n=0,\pm 1,\pm 2,\cdots $ we have 
\begin{equation}
\lim_{n\rightarrow \pm \infty }n\Delta _{n}=0.  \label{05}
\end{equation}
\end{proposition}

\begin{proof}
We have the inversion formula%
\begin{eqnarray}
\Delta _{n}=\sum_{0}^{\infty }P(S_{m}=n-1)-P(S_{m}=n) &=&  \notag \\
\sum_{0}^{\infty }\frac{1}{2\pi }\int_{-\pi }^{\pi }\hat{p}%
(t)^{m}e^{-itn}(e^{it}-1)dt &=&\frac{1}{2\pi }\int_{-\pi }^{\pi }\frac{%
(e^{it}-1)e^{-itn}dt}{1-\hat{p}(t)}.  \label{r16}
\end{eqnarray}%
Integrating by parts and noting that since everything is periodic with
period $2\pi $ the contribution from the end points cancel, gives%
\begin{eqnarray}
2\pi \Delta _{n} &=&\frac{1}{n}\int_{-\pi }^{\pi }e^{-itn}f_{1}(t)dt,\text{
with}  \label{04} \\
f_{1}(t) &=&\left( \frac{e^{it}(1-\hat{p}(t))-i(e^{it}-1)\hat{p}^{\prime }(t)%
}{(1-\hat{p}(t))^{2}}\right) .  \label{r6}
\end{eqnarray}%
Under our asumptions $\exists \delta \in (0,1)$ such that $E|X|^{1+\delta
}<\infty ,$ and this implies that%
\begin{equation*}
1-\hat{p}(t)+i\mu t\text{ }=o(|t|^{1+\delta })\text{ and }\hat{p}^{\prime
}(t)-i\mu \text{ }=o(|t|^{\delta })\text{as }|t|\rightarrow 0,
\end{equation*}%
so that%
\begin{equation*}
f_{1}(t)=o(|t|^{\delta -1})\text{ as }|t|\rightarrow 0,
\end{equation*}%
and it follows that $f_{1}$ is integrable in a neighbourhood of zero. Since $%
1-\hat{p}(t)$ is bounded away from $0$ for $|t|\in (\delta ,\pi )$, we see
that the function $f_{1}$ in (\ref{r6}) is integrable over $(-\pi ,\pi ),$
so the result follows by the Riemann-Lebesgue lemma.
\end{proof}

Recall from (\ref{r11}) that the Fourier transform of $\phi $ is given by%
\begin{equation*}
\hat{\phi}(t)=\sum_{-\infty }^{\infty }\phi _{n}e^{int}=\frac{1-\hat{p}(t)}{%
\mu (1-e^{it})},
\end{equation*}%
so we see that%
\begin{equation}
\hat{\Delta}(t):=\sum_{-\infty }^{\infty }\Delta _{n}e^{int}=\frac{e^{it}-1}{%
1-\hat{p}(t)}=-\frac{1}{\mu \hat{\phi}(t)}.  \label{r24}
\end{equation}%
Differentiating this gives%
\begin{equation}
\hat{\Delta}^{\prime }(t)=\frac{\hat{\phi}^{\prime }(t)}{\mu \hat{\phi}%
(t)^{2}}=\mu \hat{\phi}^{\prime }(t)\hat{\Delta}(t)^{2}  \label{r23}
\end{equation}%
and comparing coefficients we get%
\begin{equation}
n\Delta _{n}/\mu =\sum_{-\infty }^{\infty }m\phi _{m}\Delta _{n-m}^{(2)}%
\text{ where }\Delta _{n}^{(2)}:=\sum_{-\infty }^{\infty }\Delta _{m}\Delta
_{n-m}\text{.}  \label{06}
\end{equation}%
Note that 
\begin{equation*}
\sum_{-\infty }^{\infty }\Delta _{n}^{(2)}=\left( \sum_{-\infty }^{\infty
}\Delta _{n}\right) ^{2}=\frac{1}{\mu ^{2}}.
\end{equation*}%
We will exploit Proposition \ref{A} and (\ref{06}) to get

\begin{proposition}
\label{B} If (\ref{r1}) and (\ref{r0}) hold we have%
\begin{equation}
\lim_{n\rightarrow \infty }\frac{\Delta _{n}}{\phi _{n}^{+}}=\frac{1}{\mu },%
\text{ }\lim_{n\rightarrow -\infty }\frac{\Delta _{n}}{\phi _{n}^{+}}=0.
\label{07}
\end{equation}
\end{proposition}

\begin{proof}
From $n\Delta _{n}\rightarrow 0$ we have $\sup |n\Delta _{n}|<\infty $ so
for $n\geq 0$%
\begin{eqnarray*}
|\Delta _{n}^{(2)}| &\leq &c\{|\Delta _{0}||\Delta _{n}|+\sum_{1}^{n-1}\frac{%
1}{m(n-m)}+\sum_{n+1}^{\infty }\frac{1}{m(m-n)}+\sum_{-\infty }^{-n-1}\frac{1%
}{|m|(n-m)} \\
&\leq &\frac{c\log n}{n}\leq cn^{\delta -1},
\end{eqnarray*}%
for any $\delta >0,$ and a similar bound holds for $|\Delta _{-n}^{(2)}|,$ $%
n\geq 0.$ Taking $\delta <\beta /2$ and using this bound in (\ref{06})
together with the fact that $|\phi _{m}|\leq c|m|^{\delta -\alpha }$ for all 
$m\neq 0$ we see that%
\begin{equation*}
|\sum_{-\infty }^{\infty }m\phi _{m}\Delta _{n-m}^{(2)}|\leq c\sum_{-\infty
}^{\infty }|m|^{\delta -\beta }|n-m|^{\delta -1}\backsim c|n|^{2\delta
-\beta }\text{ as }n\rightarrow \pm \infty ,
\end{equation*}%
so that (\ref{06}) gives $|\Delta _{n}|\leq c|n|^{\delta -\alpha }$, and it
follows easily that $|\Delta _{n}^{(2)}|\leq c|n|^{\delta -\alpha }.$
Finally, with $\psi _{n}:=n\phi _{n}^{+}\in RV(-\beta )$ we write%
\begin{equation}
\frac{\Delta _{n}}{\mu \phi _{n}^{+}}=\sum_{|m|<|n|/2}\frac{m\phi _{m}}{\psi
_{n}}\Delta _{n-m}^{(2)}+\sum_{|m|\geq |n|/2}\frac{m\phi _{m}}{\psi _{n}}%
\Delta _{n-m}^{(2)}:=\sigma _{1}+\sigma _{2}.  \label{r7}
\end{equation}%
First,%
\begin{equation*}
|\sigma _{1}|\leq \frac{c|n|^{\delta -\alpha }}{\psi _{n}}%
\sum_{|m|<|n|/2}|m|\phi _{m}\backsim c|n|^{\delta -\beta }\rightarrow 0,%
\text{ as }n\rightarrow \pm \infty ,
\end{equation*}%
and by dominated convergence, which applies because $|\frac{\psi _{n-m}}{%
\psi _{n}}\Delta _{m}^{(2)}|\leq c|\Delta _{m}^{(2)}|$ on the range, 
\begin{equation*}
\sigma _{2}=\sum_{|n-m|\geq |n|/2}\frac{(n-m)\phi _{n-m}}{\psi _{n}}\Delta
_{m}^{(2)}\rightarrow \left( 
\begin{array}{ccc}
\sum_{-\infty }^{\infty }\Delta _{m}^{(2)}=\frac{1}{\mu ^{2}} & \text{as} & 
n\rightarrow \infty , \\ 
0 & \text{as} & n\rightarrow -\infty ,%
\end{array}%
\right.
\end{equation*}%
where we have used (\ref{r0}) in the second part, and the result follows.We
continue the proof of (ii) by noting that the Fourier transform of $E_{k,n}$
is given by 
\begin{eqnarray*}
\hat{E}_{k}(t) &=&\hat{u}(t)-\frac{1}{\mu (1-e^{it})}\{1+(1-\hat{\phi}%
(t))+\cdots (1-\hat{\phi}(t))^{k}\} \\
&=&\frac{(1-\hat{\phi}(t))^{k}}{\mu (1-e^{it})}\cdot \frac{1}{\hat{\phi}(t)}%
=-\psi _{k}(t)\hat{\Delta}(t),
\end{eqnarray*}%
where we have used (\ref{r24}). Thus 
\begin{equation}
E_{k,n}=-\sum_{m=-\infty }^{\infty }\Delta (m)\Psi _{n-m}^{(k+1)},\text{ }%
n=0,\pm 1,\pm 2,\cdots .  \label{r13}
\end{equation}%
and we will show that 
\begin{equation}
\lim_{n\rightarrow \infty }\frac{E_{r^{\ast }-1},n}{\Phi _{n}^{r^{\ast }-1}}%
=c(r^{\ast },\beta ),  \label{r14}
\end{equation}%
which implies our claim. We split the sum at $:=[n/2]$; we know that $\Psi
^{(r^{\ast })}$ is regularly varying so for $|m|\leq n^{\ast }\ $%
\begin{equation*}
\frac{|\Delta _{m}\Psi _{n-m}^{(r^{\ast })}|}{|\Psi _{n}^{(r^{\ast })}|}\leq
c|\Delta _{m}|\text{ and, for fixed }m,\text{ }\frac{\Delta _{m}\Psi
_{n-m}^{(r^{\ast })}}{\Psi _{n}^{(r^{\ast })}}\rightarrow \Delta _{m}.\text{ 
}
\end{equation*}%
Since $\sum_{-\infty }^{\infty }|\Delta _{m}|<\infty $ and $\sum_{-\infty
}^{\infty }\Delta _{m}=-\mu ^{-1}$ we see that%
\begin{equation*}
\sum_{|m|\in \lbrack 0,n^{\ast }]\cup \lbrack 3n^{\ast },\infty )}\Delta
_{m}\Psi _{n-m}^{(r^{\ast })}\backsim -\frac{\Psi _{n}^{(r^{\ast })}}{\mu }.
\end{equation*}%
But, using Proposition \ref{B}%
\begin{equation*}
\sum_{|m|\in (n^{\ast },3n^{\ast })}|\Delta _{m}\Psi _{n-m}^{(r^{\ast
})}|\leq c\phi _{n}^{+}\sum_{|m|\in (n^{\ast },3n^{\ast })}|\Psi
_{m}^{(r^{\ast })}|\leq c\phi _{n}^{+}
\end{equation*}%
so the result follows.

To establish (iii), we write $\phi (n)=L(n)/n^{2}$ for $n\geq 1$ and we have
that $\mu _{2}(n)=\sum_{1}^{n}m^{-1}L(m)$ is slowly varying with $L(n)/\mu
_{2}(n)\rightarrow 0,$ so that as $n\rightarrow \infty $ 
\begin{equation*}
\overline{\Phi }(n)^{2}\backsim \left( \frac{L(n)}{n}\right) ^{2}=\phi
_{n}L(n)=o(\phi _{n}\mu _{2}(n)).
\end{equation*}%
With a fixed $\delta \in (0,1/2),$ recalling that $\tsum\nolimits_{-\infty
}^{\infty }\phi _{m}=1,$ we write%
\begin{equation}
\overline{\Psi }_{n}=\overline{\Phi }(n)-\left( \phi \ast \overline{\Phi }%
\right) _{n}=\tsum\nolimits_{-\infty }^{\infty }\phi _{m}\left( \overline{%
\Phi }(n)-\overline{\Phi }(n-m)\right) .
\end{equation}%
For a fixed $\delta \in (0,1)$ it is clear that the contribution from $%
|m|>n\delta $ is $O(\overline{\Phi }(n)^{2}),$ and as $\delta \downarrow 0$%
\begin{eqnarray*}
\tsum\nolimits_{1}^{\delta n}\phi _{m}\left( \overline{\Phi }(n)-\overline{%
\Phi }(n-m)\right) &=&\phi _{n}(1+o(1))\tsum\nolimits_{1}^{\delta n}m\phi
_{m}\backsim \phi _{n}\mu _{2}(n), \\
\tsum\nolimits_{-\delta n}^{0}\phi _{m}\left( \overline{\Phi }(n)-\overline{%
\Phi }(n-m)\right) &=&o(\phi _{n}\mu _{2}(n)),
\end{eqnarray*}%
and it follows that $\overline{\Psi }(n)\backsim \phi (n)\mu _{2}(n).$ So we
know that $\overline{\Psi }(n)$ $\in RV(-2),$ and now we can repeat the
argument used in the case $\kappa \beta \in (1/2,1)$ to establish (iii).
\end{proof}

\section{The absolutely continuous case}

Unlike the lattice case, apart from assuming that $F$ has a density $f,$ we
need to make assumptions on $\hat{f}(\theta ):=\int_{-\infty }^{\infty
}f(y)e^{iy\theta }dy$ which allow to use inversion theorems. The appropriate
one is:

\textbf{Assumption} 
\begin{equation}
\text{For some }p\geq 1\text{ }\hat{f}\in L^{p}.  \label{3.0}
\end{equation}

\begin{lemma}
Write $N=\left\lfloor p\right\rfloor +2,$ and denote the density of $S_{n}$
by $f_{n},$ with $f_{1}=f.$ Then $U$ has a density $u$ which has the
representation%
\begin{equation}
u(x)=\frac{1}{\mu }\boldsymbol{1}_{[0,\infty
)}+\sum_{1}^{N}f_{n}(x)+u_{0}(x),  \label{3.2}
\end{equation}%
where $u_{0}$ is differentiable and%
\begin{equation}
xu_{0}^{\prime }(x)\rightarrow 0\text{ as }|x|\rightarrow \infty .
\label{3.1}
\end{equation}
\end{lemma}

\begin{proof}
In Lemma 2.3 of \cite{I} the representation (\ref{3.2}) is given together
with the expression%
\begin{equation}
u_{0}(x)=\frac{1}{2\pi }\int_{-\infty }^{\infty }\frac{e^{-ix\theta }\hat{f}%
(\theta )^{N+1}}{1-\hat{f}(\theta )}d\theta .  \label{3.4}
\end{equation}%
Since our assumption implies $\inf_{|\theta |\geq 1}|1-\hat{f}(\theta )|>0,$
it follows easily that%
\begin{equation}
u_{0}^{\prime }(x)=\frac{-i}{2\pi }\int_{-\infty }^{\infty }\frac{%
e^{-ix\theta }\theta \hat{f}(\theta )^{N+1}}{1-\hat{f}(\theta )}d\theta ,
\label{3.5}
\end{equation}%
and integration by parts and an appeal to the Riemann-Lebesgue Lemma gives (%
\ref{3.1}).
\end{proof}

In this scenario our aim is to give an asymptotic expansion for $u_{0}(x)$
as $|x|\rightarrow \infty ,$ the analogue of $\Delta _{n}$ is $u_{0}^{\prime
}(x),$ and (\ref{3.1}) is the analogue of Proposition 6.

A further consequence of $\inf_{|\theta |\geq 1}|1-\hat{f}(\theta )|>0$ is
that it makes it plausible that $u_{0}(x)$ can be well approximated for
large $|x|$ by%
\begin{equation*}
\frac{1}{2\pi }\int_{-1}^{1}\frac{e^{-ix\theta }d\theta }{1-\hat{f}(\theta )}%
.
\end{equation*}
This programme has been carried out to some extent in \cite{I}. In Theorem
1.1 therein, assuming that $\mu \in (0,\infty )$ and $E|S_{1}|^{\alpha
}<\infty $ for some $\alpha \in (3/2,2),$ two terms in this expansion are
identified. Their asymptotic behaviours are not discussed, but specialising
to the asymptotically stable case would yield the obvious continuous
analogues of our results for this range of $\alpha $ : it should be noted
that there are only two terms for these values of $\alpha .$

Although we will not give all details, in fact there are analogues of all
our discrete results. To state them we need to introduce probability
densities $\phi ^{\pm }$ by 
\begin{equation*}
\phi ^{+}(y)=\left\{ 
\begin{array}{ccc}
\mu _{+}^{-1}\overline{F}(y) & \text{for} & y\geq 0, \\ 
0 & \text{for} & y<0,%
\end{array}%
\right.
\end{equation*}%
and%
\begin{equation*}
\phi ^{-}(y)=\left\{ 
\begin{array}{ccc}
-\mu _{-}^{-1}F(y) & \text{for} & y<0, \\ 
0 & \text{for} & y\geq 0,%
\end{array}%
\right.
\end{equation*}%
and set%
\begin{equation*}
\phi (y)=\frac{\mu _{+}\phi ^{+}(y)}{\mu }+\frac{\mu _{-}\phi ^{-}(y)}{\mu }%
\text{.}
\end{equation*}%
Note that $\phi $ is negative on the negative axis, but $\tint\nolimits_{-%
\infty }^{\infty }\phi (y)dy=1.$ Moreover, a careful summation by parts
shows that with 
\begin{equation*}
\hat{\phi}(t):=\tint\nolimits_{-\infty }^{\infty }\phi (y)e^{iyt}dy=\frac{1-%
\hat{f}(t)}{it\mu }.
\end{equation*}%
If we denote $u_{0}^{\prime }(y)$ by $\Delta (y)$ its fourier transform is 
\begin{equation*}
\hat{\Delta}(t)=-it\hat{u}(t)=\frac{-it}{1-\hat{f}(t)}=\frac{-1}{\mu \hat{%
\phi}(t)},
\end{equation*}%
and differentiating gives%
\begin{equation}
\hat{\Delta}^{\prime }(t)=\frac{\hat{\phi}^{\prime }(t)}{\mu \hat{\phi}%
(t)^{2}}=\mu \hat{\phi}^{\prime }(t)\hat{\Delta}(t)^{2}.  \label{r60}
\end{equation}%
This is the analogue of (\ref{06}), and a very similar argument yields%
\begin{equation}
\lim_{x\rightarrow \infty }\frac{\Delta (x)}{\phi ^{+}(x)}=\frac{1}{\mu },%
\text{ }\lim_{x\rightarrow -\infty }\frac{\Delta (x)}{\phi ^{+}(x)}=0.
\label{3.3}
\end{equation}%
Next we define%
\begin{equation*}
\overline{\Phi }_{1,t}=\left\{ 
\begin{array}{ccc}
\int_{t}^{\infty }\phi (s)ds & \text{for} & t\geq 0, \\ 
-\int_{-\infty }^{t}\phi (s)ds & \text{for} & t\,<0.%
\end{array}%
\right. \text{ and }\overline{\Phi }_{k+1,t}=\overline{\Phi }_{k,t}-(%
\overline{\Phi }\ast \phi )_{t},\text{ }k\geq 1,
\end{equation*}%
and again we make the standing assumptions that $ES_{1}=\mu \in (0,\infty ),$
$ES_{1}^{2}=\infty ,$ and (\ref{r1}) and (\ref{r0}) hold. Clearly $\overline{%
\Phi }_{1,t}\backsim t\overline{F}(t)/(\beta \mu _{+})\in RV(-\beta )$ as $%
t\rightarrow \infty ,$where $\beta =\alpha -1.$ The analogue of Theorem \ref%
{D} is

\begin{theorem}
\label{E}(i) Under these assumptions if $k\geq 2$ and $k\beta <\alpha $ it
holds that as $t\rightarrow \infty $ 
\begin{equation}
\overline{\Phi }_{k,t}\backsim c(k,\beta )\left( \overline{\Phi }%
_{1,t}\right) ^{k},
\end{equation}

where $c(k,\beta )$ is given by (\ref{r53}).

(ii) If $r^{\ast }=\max (k:k\beta <\alpha )$ write 
\begin{equation*}
u_{0}(t)=\frac{1}{\mu }\left( \overline{\Phi }_{1,t}+\cdots +\overline{\Phi }%
_{r^{\ast },t}+E^{\ast }\right) .
\end{equation*}%
Then 
\begin{equation}
E^{\ast }=o(\left( \overline{\Phi }_{1,t}\right) ^{r^{\ast }})\text{ as }%
t\rightarrow \infty .
\end{equation}

(iii) If $\alpha =2$ define for $t\geq 0$ a slowly varying function by $\mu
_{2}(t)=\int_{0}^{t}y(y)dy$: then%
\begin{equation*}
u_{0}(t)-\frac{\overline{\Phi }(t)}{\mu }\backsim \frac{-\phi (t)\mu _{2}(t)%
}{\mu }\text{ as }t\rightarrow \infty .
\end{equation*}
\end{theorem}

\section{Other scenarios}

All the cases we have considered have a finite positive mean and infinite
variance and are such that the tail ratio $\rho (x):=F(-x)/\overline{F}(x)$
goes to $0$ as $x\rightarrow \infty .$ An obvious question is what can be
said if we retain all other assumptions but have $\rho (x)\rightarrow \rho
\in (0,1)?$ The only effect this has is that the constants in the crucial
estimate (\ref{r32}) change, and this is reflected in the constant in the
conclusion (\ref{r53}).

In the case that $\rho (x)\rightarrow 0$ it is clear that we need
assumptions on the left-hand tail, and if we assume $F(-x)\in RV(-\alpha )$
as $x\rightarrow \infty $ then again a modified version of (\ref{r53}) will
hold.

\end{document}